\documentclass[11pt]{article}
\usepackage{graphicx,color,amsfonts,amsmath,amssymb,enumerate}
 \usepackage{url,fancyhdr,indentfirst}
  \usepackage[colorlinks=true,citecolor=blue]{hyperref}
   \usepackage{amsthm,natbib,comment}
\usepackage[english]{babel}
\usepackage{tikz-qtree}
\def\d{\,\mathrm{d}}

\newcommand{\E}{\mathbb{E}}

\newcommand{\R}{\mathbb{R}}
\newcommand{\N}{\mathbb{N}}
\newcommand{\p}{\mathbb{P}}

\newcommand{\U}{{\rm U}}

\newcommand{\lcx}{\prec_{\mathrm{cx}}}

\renewcommand{\ge}{\geqslant}
\renewcommand{\le}{\leqslant}

\renewcommand{\epsilon}{\varepsilon}

\renewcommand{\cdots}{\dots}

\renewcommand{\prec}{\preccurlyeq}

\def\D{\mathcal D}
\def\C{\mathcal C}

\newtheorem{theorem}{Theorem}
\newtheorem{corollary}[theorem]{Corollary}
\newtheorem{lemma}[theorem]{Lemma}
\newtheorem{proposition}[theorem]{Proposition}
\theoremstyle{definition}

\newtheorem{remark}{Remark}[section]

\numberwithin{equation}{section}
  \numberwithin{theorem}{section}
    \allowdisplaybreaks
\usepackage{setspace} 
\usepackage{footmisc}
\renewcommand{\cite}{\citet}

\usepackage[compact]{titlesec}
\titlespacing{\section}{0pt}{8pt}{4pt}
\titlespacing{\subsection}{0pt}{6pt}{4pt}
\titlespacing{\subsubsection}{0pt}{5pt}{3pt}

\begin{document}

\title{Sums of Standard Uniform Random Variables}

\author{Tiantian Mao\thanks{Department of Statistics and Finance, University of Science and Technology of China, Hefei, Anhui 230026, China. Email: \texttt{tmao@ustc.edu.cn}.}\and Bin Wang\thanks{Corresponding author. RCSDS, NCMIS, Academy of Mathematics and Systems Science,
Chinese Academy of Sciences, Beijing 100190, China. Email: \texttt{wangbin@amss.ac.cn}} \and Ruodu Wang\thanks{Department of Statistics and Actuarial Science, University of Waterloo, Waterloo, ON N2L3G1, Canada. Email: \texttt{wang@uwaterloo.ca}.}}

\date{\today}

\maketitle

\begin{abstract}
In this paper, we analyze the set of all possible aggregate  distributions of the sum of standard uniform random variables,
a simply stated yet challenging problem in the literature of distributions with given margins. Our main results are obtained for two distinct cases.
In the case of dimension two, we obtain four partial characterization results.
For dimension greater or equal to three, we obtain a full characterization of the set of aggregate distributions, which is the first complete characterization result of this type in the literature for any choice of continuous marginal distributions.

~

\end{abstract}

\noindent\rule{\textwidth}{0.5pt}

\section{Introduction}

 Many questions remain open in the determination of probability measures with
given margins and other constraints, since the seminal work of \cite{S65}.
One of the challenging and recently active questions is, for $n$ given distributions $F_1,\cdots, F_n$ on $\R$, to determine all
possible distributions of $S_n = X_1 + \cdots + X_n$ where $X_1,\dots,X_n$ are random variables with respective distributions $F_1,\dots,F_n$ in an atomless probability space $(\Omega,\mathcal A,\p)$.
Formally, denote the set of possible distributions of the sum
$$
\mathcal D_n = \mathcal D_n(F_1,\ldots,F_n) = \left\{{\rm cdf~ of}~ X_1 + \cdots + X_n : X_i \sim F_i,\, i = 1,\ldots, n\right\}.
$$
where $X \sim F$ means that the cdf of a random variable $X$ is $F$.

The question of  characterizing $\mathcal D_n$, although simply stated, is a challenging open question.
Generally, it is not easy to determine whether a given distribution $G$ is in $\mathcal D_n$, although many moment inequalities can be used as necessary conditions.
In the recent literature, some papers partially address the question of $\mathcal D_n$ and provide sufficient conditions for $G\in \mathcal D_n$; see \cite{BJW14}, \cite{MW15} and \cite{WW16}.
A particular question is whether a point-mass belongs to $\mathcal D_n$, which is referred to the problem of joint mixability (\cite{WPY13}), and has found many applications in optimization and risk management.

Yet, there are no known results on the characterization of $\mathcal D_n$, except for the trivial case where each of $F_1,\dots,F_n$ is a Bernoulli distribution  in a low dimension.
Perhaps, the attempt to characterize $\mathcal D_n$ for generic  $F_1,\dots,F_n$ is too ambitious. In this paper,
we focus on the very special case where $F_1,\dots,F_n$ are standard uniform distributions $\rm U[0,1]$.
This problem might look naive at the first glance, but with  the technical challenges we shall see in this paper, the characterization of $\mathcal D_n$ is highly non-trivial even for uniform distributions.

As a well-known fact, for any random variables $X_1,\dots,X_n$, their sum is dominated in \emph{convex order} by $X_1^c+\dots+X_n^c$ where $X^c_i$ is identically distributed as  $X_i$, $i=1,\dots,n$, and $X^c_1,\dots,X^c_n$ are \emph{comonotonic}. Hence, the set  $\mathcal D_n(F_1,\ldots,F_n) $ is contained in the set $\mathcal C_n$ of distributions that are dominated in convex order by the distribution of the comonotonic sum (for the precise definitions, see Section \ref{sec:2}).
These two sets are asymptotically equivalent after normalization as shown by \cite{MW15}. One naturally wonders whether they coincide for a finite $n$.

The main results of this paper can be summarized below in two distinct cases.
Recall that $F_1,\dots,F_n$ are  standard uniform distributions $\rm U[0,1]$.
For dimension $n=2$, the sets $\mathcal D_n$ and $\mathcal C_n$ are not equivalent, and a complete determination of $\mathcal D_2$ for uniform margins is still unclear.
In Section \ref{sec:3}, we provide four results on equivalent conditions for various types of distributions to be in $\D_n$,  including unimodal, bi-atomic, and tri-atomic distributions, and distributions dominating a proportion of a uniform one.
In Section \ref{sec:4},  we are able to analytically characterize the set $\mathcal D_n$ for dimension $n\ge 3$ by showing that $\mathcal D_n=\mathcal C_n$.
 This result came as a pleasant surprise to us, since it is well known that the dependence structure gets much more complicated as the dimension grows.
 As far as we are aware of, this result is the first full characterization of $\mathcal D_n$ for any types of continuous marginal distributions.

For applications of the problem of possible distributions of the sum with specified marginal distributions,
  we refer to \cite{EPR13}, \cite{R13}, \cite{MFE15} and \cite{BBV18}. A particular problem on the sum of standard uniform random variables  is the aggregation of p-values from multiple statistical tests,  which are uniform by definition under the null hypothesis. These p-values, as obtained from different tests, typically have an unspecified  dependence structure, and hence it is important to understand the possible distributions of the aggregated $p$-value; see \cite{VW18}.
  
  We remark that the determination of whether $\mathcal D_n=\mathcal C_n$ for general margins is unclear. 
  Note that the determination of $\mathcal D_n=\mathcal C_n$  for a given tuple of margins requires more than the determination of joint mixability of the margins, and the latter is known to be an open question in general. 
  We conjecture that for general margins on bounded intervals  with sufficient smoothness, there is a dimension $n$ above which $\mathcal D_n$ equals $\mathcal C_n$, but this is out of reach by current techniques, as our proofs heavily rely on the specific form of uniform distributions.

\section{Preliminaries and notation}\label{sec:2}

In this paper, for any (cumulative) distribution function $F$, we denote by $F^{-1}(t) = \inf\{x :F(x) \ge t\}, ~t \in (0, 1]$, the quantile function of $F$. Denote by $\mathcal X$ the set of integrable random variables and $\mathcal F$ the set of distributions with finite mean. The terms ``distributions" and ``distribution functions" are treated as identical in this paper. For $F \in \mathcal F$, ${\rm Supp}(F)$ is the essential support of the distribution measure induced by  $F$, which will be referred to as the support of $F$.
For any distributions $F$, $G \in \mathcal F$, we denote by $F\oplus G$ the
distribution with quantile function $F^{-1}(t) + G^{-1}(t)$, $t\in [0,1]$. For a distribution $F$, $\mu(F)$ denotes the expectation of $F$. 
%
 Throughout ${\lceil x \rceil}$ and ${\lfloor x \rfloor}$ 
 represent for the ceiling and the  floor   
 of $x\in \R$, respectively.
 A  density function $f$ is  \emph{unimodal} if there exists   $a\in \R$ such that 
 $f  $  is increasing on $(-\infty, a] $ and decreasing on $[a,\infty)$.


The set $\mathcal D_n$ is related to the notion of convex order.
A distribution $F\in \mathcal F$ is \emph{smaller than $G\in\mathcal F$ in convex order}, denoted by  $F \prec_{\rm cx} G$,  if
$$
\int_\R   \phi(x)\d F(x) \le\int_\R\phi(x) \d G(x) ~~{\rm for~ all~ convex}~\phi: \R\to \R;
$$
provided that both expectations exist (finite or infinite). Standard references for convex order are  \cite{MS02} and \cite{SS07}.
For a given distribution $F \in \mathcal F$, denote by $\mathcal {C}(F)$ the set of all distributions dominated by
$F$ in convex order, that is,
$$
 \mathcal C(F) = \{G \in \mathcal F : G\prec_{\rm cx} F\}.
$$
Checking whether a distribution $G$ is in $\mathcal C(F)$ can be   conveniently done using  an equivalent condition (e.g.~Theorem 3.A.1 of \cite{SS07}) is
\begin{align}\label{eq:eqcx}
\inf_{k\in \R} \left(\int_\R (x-k)_+ \d F(x) -\int_\R (x-k)_+\d G(x)\right)\ge 0.
\end{align}

As the focus of this paper is the distribution of the sum of uniform random variables, we use the following simplified notation.
For $n\in \N$ and $x\in \R_+$,
 write
$$
   \mathcal D^{\rm U}_n   = \mathcal D_n({\rm U}[0,1],\ldots,{\rm U}[0, 1]) \mbox{~~~and~~~}
    \mathcal C^{\rm U}_x = \mathcal C({\rm U}[0, x]),
$$
where ${\rm U}[a,b]$ stands for the uniform distribution over an interval $[a,b]\subset \R$.

Below we list some basic properties of the sets $\mathcal D_n(\cdot)$ and $\mathcal C(\cdot)$; certainly, they hold also for
$\mathcal  D^{\rm U}_n$ and $\mathcal  C^{\rm U}_x$. First, note that if all distributions $F_1,\ldots,F_n$ are shifted by some constants or
scaled by the same positive constant, then the elements in $\mathcal D_n(F_1,\ldots,F_n)$ are also simply shifted
or scaled. Moreover, $\mathcal D_n(F_1,\ldots,F_n)$ is symmetric in the distributions $F_1,\ldots,F_n$. These facts
allow us to conveniently exchange the order of the distributions $F_1,\ldots,F_n$ and normalize these
distributions by shifts and a common scale. 
For given distributions $F_1,\ldots,F_n \in\mathcal F$, {the distribution $F_1\oplus \cdots \oplus F_n$ is the maximum in convex order in the set $\mathcal D_n(F_1,\dots,F_n)$.}
 This fact is summarized in the following lemma.
\begin{lemma}\label{lm-cxset-1}
For $F_1,\ldots,F_n\in\mathcal F$, $\mathcal D_n(F_1,\ldots,F_n)\subset \mathcal C(F_1\oplus \cdots \oplus F_n)$.
\end{lemma}

Lemma \ref{lm-cxset-1} can be equivalently stated as the following. {If $X_1\sim F_1,\ldots, X_n \sim F_n$ and
$F$ is the distribution of $X_1+\cdots+X_n$, then $F\prec_{\rm cx} F_1 \oplus \cdots  \oplus F_n$.} For a history of this result, see,
for instance, \cite{PW15}. In particular, if $F \in \mathcal D_n(F_1,\ldots,F_n)$, then the mean of
$F$ is fixed and is equal to the sum of the means of $F_1,\ldots,F_n$.

In view of  Lemma \ref{lm-cxset-1}, it would be  natural to investigate when the two sets coincide,
that is, $\mathcal D_n (F_1,\dots,F_n) = \mathcal C(F_1\oplus\dots\oplus F_n)$.
Note that for a given $G\in \mathcal F$, the determination of $G\in \mathcal C(F)$ can be analytically checked with its equivalent condition \eqref{eq:eqcx}.
 Hence, if the above two sets coincide, then we have an analytical characterization of  $\mathcal D_n (F_1,\dots,F_n)$.
In the case of uniform distributions, one wonders whether $\D^{\rm U}_n = \C^{\rm U}_n$, noting that $\D^{\rm U}_n \subset \C^{\rm U}_n$ always holds.  Unfortunately, as shown in \cite{MW15}
by a counter-example, in the simple case $n=2$, $\D_2^{\rm U}$ is an essential subset of $\C^{\rm U}_2$; see Theorem \ref{th-bi} in Section \ref{sec:3} for distributions in $\C^{\rm U}_2$ but not in  $\D_2^{\rm U}$.


Some basic properties of the set $\D_n(\cdot)$ are given in the following lemma.
\begin{lemma} \label{lm-180905-1}For any $F_1,\ldots, F_n\in \mathcal F$, the set $\D_n(F_1,\ldots,F_n)$ is non-empty, convex and closed with respect to weak convergence.
\end{lemma}
Another simple fact is that a distribution in $\D_n(F_1,\ldots, F_n)$ has to have the {correct support}
generated by the marginal distributions.
 That is, for any $F_1,\ldots, F_n\in \mathcal F$ and $G\in \D_n(F_1,\ldots,F_n)$, we have ${\rm Supp}(G)\subset \sum_{i=1}^n {\rm Supp}(F_i)$.
For more properties on the sets $\D_n(\cdot)$ and $\mathcal C(\cdot)$, see \cite{MW15}.

%
A useful concept for our analysis of $\mathcal D_n(\cdot)$ is the joint mixability introduced by  \cite{WPY13}.
  An $n$-tuple of distributions $(F_1,\ldots,F_n)\in \mathcal F^n$ is said to be \emph{jointly mixable} (JM), if  $  \mathcal D_n(F_1,\dots,F_n)$ contains a point-mass $\delta_K$, $K\in \R$. 
  This point-mass is unique if the distributions $F_1,\dots,F_n$ have bounded supports.
If $(F_1,\ldots,F_n)$ is JM and $F_1=\cdots=F_n=F$, then we say that $F$  is $n$-\emph{completely mixable} ($n$-CM). 


\section{Sums of two standard uniform random variables}\label{sec:3}

In this section, we look at the sum of two ${\rm U}[0, 1]$ random variables. Unfortunately, as explained above, a full characterization of $\D_2^{\rm U}$
appears difficult to obtain. In this section, we provide characterization results for four different types of distributions to be in $\D_2^{\rm U}$. We first present the main results on $\D_2^{\rm U}$ in Section \ref{sec:sec3.1}.  Their proofs will be given in Section \ref{sec:sec3.2}.
Before presenting our main findings, we summarize some existing results on $\D_2^{\rm U}$. These facts
can be derived from existing results on joint mixability in \cite{WW16}.

\begin{proposition} \label{pr-firstobs}
We have
\begin{itemize}
\item [(i)] $\D_2^{\rm U}\subsetneq \C^{\rm U}_2$.
\item [(ii)] Let $F$ be any distribution with a monotone density function on  ${\rm Supp}(F)$. Then $F\in \D_2^{\rm U}$ if and only if  ${\rm Supp}(F)\subset [0,2]$ and $F$ has mean $1$.
\item [(iii)] Let $F$ be any distribution with a unimodal and symmetric density function on ${\rm Supp}(F)$. Then $F\in \D_2^{\rm U}$ if and only if ${\rm Supp}(F)\subset [0,2]$ and $F$ has mean $1$.
\end{itemize}
\end{proposition}

\begin{remark}
For a uniform distribution on an interval of length $a$, ${\rm U}[1-\frac a 2,1+\frac a 2]\in\mathcal D_2^{\rm U}$ if and only if $a\in [0,2]$, a special case of (ii) and (iii) of Proposition \ref{pr-firstobs}.  The case ${\rm U}[\frac 12,\frac 32]\in \mathcal D_2^{\rm U}$ is shown by \cite{R82}, and the general case $a\in [0,2]$ is shown by \cite{WW16}.
\end{remark}

\subsection{Main results} \label{sec:sec3.1}


As a first new result in this paper, we show that the class of distributions with a unimodal density with the correct mean is contained in $\D_2^{\rm U}$.

\begin{theorem}\label{main-th} Let $F$ be a distribution with a unimodal density on $[0, 2]$ and mean $1$. Then $F\in \D_2^{\rm U}$.
\end{theorem}

Both the two previous results in Proposition \ref{pr-firstobs} (ii) and (iii) are special cases of Theorem \ref{main-th};
thus Theorem \ref{main-th} generalizes the existing results derived from \cite{WW16}.

The second result of the paper concerns the class of
distributions which dominate a proportion of a uniform distribution.

{\begin{theorem}\label{pr-contines}
Let $F$ be a distribution supported in $[a,a+b]$ with mean $1$ and density function $f$. If there exists $h>0$ such that $f>3b/(4h)$ on $[1-h,1+h]$, then $F\in { \D}_2^{\rm U} $.
\end{theorem}

In Theorem \ref{pr-contines}, the condition that the density of $F$ dominates $3b/2$ times that of ${\rm U}[1-h,1+h]$ immediately  implies the following admissible ranges 
of $a$, $b$ and $h$: 
$a\ge 1/3$, $b\le 2/3$, and $h\le 1/3$. 
Hence, $\mathrm{Supp}(F)\subset [0,2]$, which is obviously necessary for $F\in   \D_2^{\rm U}(1, 1)$.
Theorem \ref{pr-contines} also  immediately implies the following fact: For  a distribution $F$ with mean $0$ and bounded support, if $F$ has a positive
density $f > \epsilon$ in a neighbourhood of $0$ for some $ \epsilon> 0$, then $F \in \D_2(\mathrm U[-m,m],\mathrm U[-m,m])$ for sufficiently large
$m > 0$.


In addition to the two results on continuous distributions, we analyze discrete distributions.
We shall obtain two results, one characterizing bi-atomic distributions in $\D_2^{\rm U} $, and one characterizing equidistant tri-atomic distributions in  $\D_2^{\rm U} $.

 \begin{theorem} \label{th-bi}Let $F$ be a bi-atomic distribution with mean $1$ supported on $\{a,a + b\}$ with $b>0$. Then $F \in \D^{\rm U}_2  $ if and only if $1/b \in
 \N$.
\end{theorem}

For a given $a\in [0,1)$ and $a+b\in (1,2]$, there is a unique distribution on $\{a,a+b\}$ with mean $1$. Hence, all the bi-atomic distributions that belong to $\D^{\rm U}_2  $
have the corresponding distribution measures
$$
 \Big\{\nu :~~1-\nu(\{a\})=\nu(\{a+1/k\}) = (1-a)k,~k\in\N,~a\in [0,1)\Big\}.
$$
Note that  many bi-atomic distributions supported on $\{a,a+b\}$ are in $\C^{\rm U}_2$ but not in  $\D_2^{\rm U} $, as long as $1/b\not\in \N$.
For example, one can choose a bi-atomic distribution $F$  with equal probability on $\{1-1/\pi,1+ 1/\pi\}$, and easily see that $F\in\C^{\rm U}_2$, whereas from Theorem \ref{th-bi}  we find that   $F$ is not in $\D_2^{\rm U} $.
Thus, Theorem \ref{th-bi} implies   $\D_2^{\rm U} \subsetneq \C^{\rm U}_2$, a fact as noted by \cite{MW15}.

For a tri-atomic distribution $F$, write $F = (f_1, f_2, f_3)$ where $f_1, f_2, f_3$ are the probability
masses of $F$. Note that on given three points, the set of tri-atomic distributions with mean $1$
has one degree of freedom. For tractability, we study the case of $F$ having an \emph{equidistant support} in the form of
$\{a-b,\,a,\,a+b\}$ for some $b>0$. We only consider the case $b\le a \le 1$ since the case $a > 1$ is symmetric.


To state our characterization of tri-atomic distribution in  $\mathcal D_2^{\mathrm U}$, we introduce the following notation.
For $x > 0$, define a \emph{measure of non-integrity}
$$
  \lceil x \rfloor = \min \left\{\frac{\lceil x \rceil}{x} - 1, 1 - \frac{\lfloor x \rfloor}{x}\right\}\in [0, 1].
$$
Obviously $\lceil x \rfloor=0 \Leftrightarrow  x \in \N$.
\begin{theorem}\label{th-tri}
Suppose that $F = (f_1, f_2, f_3)$ is a tri-atomic distribution with mean $1$ supported
in $\{a - b, a, a + b\}$  and $0<b\le a \le1$. Then $F \in \mathcal D_2^{\mathrm U} $ if and only if it is the following three cases.
\begin{itemize}
\item [(i)] $a=1$ and $f_2\ge \lceil \frac1{2b}\rfloor$.
\item [(ii)] $a<1$ and $\frac1{2b}\in \N$.
\item [(iii)] $a<1$, $\frac1{2b}-\frac12\in \N$ and $f_2\ge a+b-\frac12$.
\end{itemize}
\end{theorem}


The corresponding distributions in Theorem \ref{th-tri} are summarized below. Write $c = a/b +1- 1/(2b)$.
${\rm cx}(x, y)$ stands for the convex set generalized by some vectors $x,y$.
\begin{itemize}
\item [(i)] $(f_1,f_2,f_3)\in {\rm cx}\{(0,1,0), \frac12(1-\lceil \frac1{2b} \rfloor, 2\lceil \frac1{2b} \rfloor, 1- \lceil \frac1{2b} \rfloor)\}$.
\item [(ii)] $(f_1,f_2,f_3)\in {\rm cx}\{(0,c,1-c), \frac12(c, 0, 2-c)\}$.
\item [(iii)] $(f_1,f_2,f_3)\in {\rm cx}\{(0,c,1-c), \frac12(c(1-b), 2bc, 2 - c(1-b))\}$.
\end{itemize}

\subsection{Proofs of the main results} \label{sec:sec3.2}
To prove Theorem \ref{main-th}, we need the following lemma.
To state it, we introduce the notion of special simple unimodal functions. A function $h$ is called a \emph{special simple unimodal (SSU) function on $[a,a+n)$}, $a\in\R$, $n\in\N$,  if it is  unimodal, and $h(x)$ is a constant on  $[a+k-1,a+k)$ for each $k=1,\ldots,n$.

\begin{lemma}\label{lm-171231-1}
Let $F$  be a distribution function with density function $f$ and support $[a,a+n)$, $a\in\R_-$, $n\in\N$. Suppose that $F$ has   mean $0$ and $f$ is a SSU function on  $[a,a+n)$.  Then for $c\ge \max\{a+n,-a\}/2$, we have $F\in \D_2({\rm U}[-c,c],{\rm U}[-c,c])$.
\end{lemma}
\begin{proof}
We show the result by induction. For $n=1$,  then we have $F$ is the uniform distribution on $[a,a+1)$ with mean $0$. This means $a=-0.5$, that is, $F$ is the distribution of ${\rm U}[-0.5,0.5]$, and $c\ge 0.5$. By Theorem 3.1 of \cite{WW16}, we know that  ${\rm U}[-0.5,0.5]$,  ${\rm U}[-c,c]$ and ${\rm U}[-c,c]$ are jointly mixable as the mean inequality $-0.5-2c+2c\le 0 \le 0.5+2c-2c$ is satisfied. 
This means $F\in \D_2({\rm U}[-c,c],{\rm U}[-c,c])$.

Next, we assume the result holds for $n\le k$ and show it holds for $n=k+1$. 
 Without loss of generality, we assume $a+n>-a$. Otherwise consider the distribution of $X^*=-X$ with $X\sim F$.
Define a  distribution $H$ with  density function $h: [a,a+n)\to\R_+$ defined as
$$
 h(x)= \frac{n+1+2a}{n}=:\alpha,~x\in [a,a+1),~~~h(x) = \frac{-1-2a}{n(n-1)}=:\beta,~~x\in[a+1,a+n).
$$
It can be easily verified that it has mean $0$ and  $\alpha\ge 1/n\ge \beta>0$ as $2a+n>0$. That is, $h$ is a decreasing density function on $[a,a+n)$ with mean $0$. Then by  Theorem 3.2 of \cite{WW16}, we have $H$, ${\rm U}[-c,c]$ and ${\rm U}[-c,c])$ are jointly mixable.  Hence, we have $H\in \D_2({\rm U}[-c,c],{\rm U}[-c,c])$.

Denote $a_i = f(a+i)$, $i=0,\ldots, n-1$. Since $f$ is unimodal, without loss of generality, assume that $a_1\le \cdots\le a_k\ge \cdots \ge a_n$ for some $k\in \{0,\ldots, n-1\}$. Let
$$
  \lambda := \min\left\{ \frac{a_1}{\alpha},\frac{a_2}{\beta},\ldots, \frac{a_n}{\beta}\right\}.
$$
 By contradiction, it immediately follows  from  $\sum_{i=1}^na_i=1$ and $\alpha+(n-1)\beta=1$ that  $\lambda \le 1$.   If $\lambda=1$, then $a_1=\alpha$ and $\alpha_i=\beta$, $i=2,\ldots,n$, that is, $F=H\in \D_2({\rm U}[-c,c],{\rm U}[-c,c])$, which shows the statement in the lemma.  Next, we consider the case $\lambda<1$. We first assert that the following sequence
$$
  b_1 =a_1-\lambda \alpha,~ ~~b_i = a_i-\lambda \beta,~~i=2,\ldots,n.
$$
is unimodal such that either $b_1$ or $b_n$ is $0$.  To see it, we only need to show it is unimodal by the definition of $\lambda$. We consider the following two cases.
\begin{itemize}
\item [(i)]  If $k\ge 2$, then $a_1\le a_2$ and hence, $b_1=a_1-\lambda \alpha \le a_2 -\lambda \beta =b_2$ as $\alpha\ge \beta$. Also, note that $b_2\le \cdots\le b_k\ge \cdots \ge b_n$. Hence, the sequence $\{b_i,\,i=1,\ldots,n\}$ is unimodal.
\item [(ii)] If $k=1$, then $a_1\ge a_2\ge \cdots\ge a_n$, and hence, $b_2\ge \cdots\ge b_n$.  No matter $b_1\ge b_2$ or $b_1\le b_2$, we have that $\{b_i,\,i=1,\ldots,n\}$ is unimodal.
\end{itemize}
Define $
   F_0 = \frac{F-\lambda H}{ 1-\lambda}.
$
{It is easy to check that $F_0$ is a distribution function with a density function $f_0$ taking value $b_i/(1-\lambda)$ on the interval $[a,a+i)$, $i=1,\ldots,n$. 
By the above observations on  the  sequence $\{b_i,\,i=1,\ldots,n\}$,}  we know that   $F_0$ is a distribution with mean $0$ and with support on a subset of $[a,a+n-1)$ or $[a+1, a+n)$.  Then by induction, we have $F_0\in  \D_2({\rm U}[-c,c],{\rm U}[-c,c])$. 
Since $F = (1 -\lambda)F_0 + \lambda H$, and    $\D_2$ is a convex set, we have $F\in  \D_2({\rm U}[-c,c],{\rm U}[-c,c])$.
\end{proof}

Now we are ready to prove Theorem \ref{main-th}.

\begin{proof}[Proof of  Theorem \ref{main-th}]
Denote by $f$ the density function of $F$, which is unimodal.  Then there exists  $x_0\in [0,2]$ such that $f(x_0)=\max\{f(x),x\in[0,2]\}$.   Let $X$ and $Y$ be  two independent random variables  such that $X\sim F$ and $Y\sim {\rm U}[0,1]$, and  define $X_m=\lfloor m(X-1)\rfloor +Y$, $m\in\N$.
Then the density function  of $X_m$, denoted by $h_m$, is   given by
 $$h_m(x)=\int_{j/m}^{(j+1)/m} f(x) \d x=:p_j, ~~x\in [j-m,j-m+1),~j=0,\ldots,2m-1.$$    Let $k=\lfloor mx_0\rfloor$. Then we have
$$
   p_0\le p_1\le \cdots\le p_{k-1},~~~p_{k+1}\ge p_{k+2}\ge \cdots\ge p_{2m},
$$
and by $f(x)$ is unimodal on $[k/m,(k+1)/m]$,
$$
   p_k=\int_{k/m}^{(k+1)/m} f(x) \d x \ge  \frac1m\min \left\{ f(k/m),\,f((k+1)/m) \right\} \ge \min\{p_{k-1},\,p_{k+1}\}.
$$
That is, the sequence $\{p_k,\,k=1,\ldots,2m\}$ is unimodal. It can be verified that $$\E[X_m]=\E[\lfloor m(X-1)\rfloor-m(X-1)]+ \frac 12 ,$$  which lies in $[-0.5,0.5]$.  Then the random variable  $Y_m:=X_m-\mu(X_m)$  has mean $0$ and takes value in a subset of
$[-m-0.5,m+0.5]$. By Lemma  \ref{lm-171231-1}, we have the distribution of $Y_m$ belongs to $\D_2({\rm U}[-c,c],{\rm U}[-c,c])$ with {$c=m/2+0.25$}.
Then obviously, we have the distribution of $Y_m/(2m+0.5)+1$ belongs to $\D_2^\U $. Note that
$Y_m/(2m+0.5)+1$ converges to $X$ in $L^\infty$-norm as $m\to\infty$. Hence, by the closure of $\D_2^\U $ with respect to $L^\infty$-norm, we have $F\in\D_2^\U $. This completes the proof.
\end{proof}

To prove Theorems \ref{pr-contines} - \ref{th-tri},  we need the following lemma.

\begin{lemma}\label{lm-180228-1}
{Let  $Z$ be a random variable with distribution $F$   supported in $\{b-k,\, b-k+1,\ldots,b\}$, $k\in\N$, $b\in\R$,
}  satisfying
$$
\p(Z=b-i)=p_i\ge 0,~~i=0,\ldots,k,~~and~~\sum_{i=0}^k p_i=1.
$$ 
{Then we have the following statements hold. }
\begin{itemize}
\item [(i)] If $F\in \D_2({\rm U}[0,T],{\rm U}[-T,0])$, then  at least  one of $b$ and $T$ is an integer.
\item [(ii)] If $k=1$, then $F\in \D_2({\rm U}[0,T],{\rm U}[-T,0])$ if and only if $b\in (0,1)$ and $T\in\N$.
\item [(iii)] If $k=2$, then $F\in \D_2({\rm U}[0,T],{\rm U}[-T,0])$ if and only if  one of the following three cases holds:
   \begin{itemize}
\item [(a)] $b=1$, $p_2=p_0=(1-p_1)/2$, $p_1\ge r/T$ with $T=2m\pm r$ with $m\in\N$ and $r\in (0,1)$.  
\item [(b)]  $b\in (0,1)\cup (1,2)$,  $T$ is even, $p_1\in (0,\min\{b,2-b\})$, $p_2=(b-p_1)/2$ and $p_0=1-(b+p_1)/2$.

\item [(c)]  $b\in (0,1)\cup (1,2)$,  $T$ is odd, $p_1\in (b^*/T,b^*)$ with $b^*=\min\{b,2-b\}$, $p_2=(b-p_1)/2$  and $p_0=1-(b+p_1)/2$.
\end{itemize}
\end{itemize}
\end{lemma}
\begin{proof}
First we introduce the notation:
 for any random variable $X$ and any set $L\subset \R$, define   random events
$$
  A_X(L) :=\{X-n\in L~{\rm for~some}~n\in\N\}=\{X\mod 1\in L\},
$$
and a function
$$
g_X(x) = \lim_{\delta\downarrow 0} \frac{\p(A_X([x-\delta,x+\delta]))}{2\delta}
,~~x\in\R,~~\mbox{if the limit exists}.
$$
Note that $Z \equiv b \mod 1$, $\E[Z]=0$ and hence, $b-k\le 0\le b$.  If $F\in \D_2({\rm U}[0,T],{\rm U}[-T,0])$, then there exist two random variables   $X\sim {\rm U}[0,T]$ and $Y\sim {\rm U}[-T,0]$ such that $Z=X+Y$ a.s.
Since $Z=X+Y \equiv b \mod 1$ a.s., we have $A_X(L)=A_Y(b-L)$ a.s. for any $L\subset \R$, where $b-L=\{b-x: x\in L\}$. 
We first show that at least  one of $b$ and $T$ is an integer. To this end, assume that $T$ is not an integer. Then there exists $\ell \in \N$ such that $T=\ell+t$ with $t\in(0,1)$. Note that $X\sim {\rm U}[0,T]$. We have
\begin{align*}
  g_X(x) & =\lim_{\delta\downarrow 0} \frac{\p(X\mod 1\in [x-\delta,x+\delta])}{2\delta}\\
  & = \frac1{2T} \left(\#\{n\in\N: n+x\in (0,T)\} +\#\{n\in\N: n+x\in [0,T]\} \right)\\
  & = \begin{cases}
  \frac{\ell+1}{T},& {\rm if}~x\mod 1\in (0,t),\\
   \frac{\ell+1/2}{T},&  {\rm if}~x\mod 1= 0~{\rm or}~t,\\
  \frac \ell {T}, &  {\rm if}~x\mod 1\in (t,1).
  \end{cases}
\end{align*}
 {By the definition of $A_X(L)$ and $Y\stackrel{\rm d}= -X$, we have  $A_Y([x-\delta,x+\delta]) =A_{-X}([x-\delta,x+\delta])=  A_{X}([-x-\delta,-x+\delta])$ a.s., and thus, $g_Y(x)=g_X(-x)$ for $x\in\R$. Also, note that $A_X(L)=A_Y(b-L)$ a.s. which implies  $g_X(x)=g_Y(b-x)$. } Therefore, we have $$g_X(x)=g_Y(b-x)=g_X(x-b)~~{\rm for~any~}x\in\R.$$ 
Note that there exists $x\in\R$ such that $x\mod 1\in (0,t)$ and $x-b\mod 1\in (t,1)$ which contradicts with the formula of $g_X$. Hence, $b$ must be an integer.

Next we consider the two cases that $k=1$ and $k=2$.
\begin{itemize}
\item [(i)] For $k=1$, note that for any $b\in\N$, $\mu(F)$ could not be $0$. Hence, we only need to show $F\in \D_2({\rm U}[0,T],{\rm U}[-T,0])$ when $b\in (0,1)$ and $T\in \N$.
     By $\mu(F)=0$, we have $p_1=b$ and $p_0=1-b$.
Let  $X\sim {\rm U}[0,T]$  and  define random variable $Y$ such that 
$$
[Y|X\in[k,k+b]] = b-1-X,~ {\rm a.s.~~for} ~ k=0,\ldots,T-1
$$
and
$$
[Y|X\in(k+b,k+1]] = b-X,~ {\rm a.s.~~for} ~ k=0,\ldots,T-1.
$$
Then it is easy to see that $Y\sim {\rm U}[-T,0]$ and $X+Y$ has the distribution $F$.
 %
 Thus, we have $F\in\mathcal D_2({\rm U}[0,T], {\rm U}[-T,0])$.

\item [(ii)] If $k=2$, by the necessity condition that at least  one of $b$ and $T$ is an integer, we consider the following three cases. Without loss of generality, assume $p_i=\p(Z=b-i)>0$, $i=0,1,2.$
\begin{itemize}
\item [(a)] If $b$ is an integer,  for the mean-constraint to be satisfied, we have $b=1$  and
    $$
    \p(Z=-1) =\p(Z=1)=p_0>0~~{\rm and}~~\p(Z=0)=p_1 =1-2p_0.
    $$
    Let $T= 2m\pm r$ with $r\in [0,1]$ and $m\in\N$.
    We only need to find the smallest value of $p_1$ such that $F\in \D_2({\rm U}[0,T],{\rm U}[-T,0])$ as $\delta_0\in \D_2({\rm U}[0,T],{\rm U}[-T,0])$ and $\D_2({\rm U}[0,T],{\rm U}[-T,0])$ is closed under mixture, where $\delta_0$ is the point-mass at  $0$. Assume that there exist $X\sim {\rm U}[0,T]$  and  $Y\sim {\rm U}[-T,0]$ such that $Z=X+Y$ a.s.
    If $T=2m+r$, then since $b$ is an integer, we have 
\begin{align*}
      A_X((r,1)) &= \big\{X\in \cup_{k=1}^{2m} (r+k-1,k)\big\}\\ 
      &= A_Y(\{(-1,-r)\})= \big\{Y\in \cup_{k=1}^{2m} (-k, -r-k+1)\big\}~~{\rm a.s.} 
\end{align*}
We let
\begin{align*}
     [Y|X\in A_1] = -X-1~~{\rm and}~~[Y|X\in A_2] = -X+1~~{\rm a.s.}
\end{align*}
    where $A_1=\cup_{k=1}^{m} (r+2k-3,2k-1)$ and $A_2= \cup_{k=1}^{m} (r+2k-1,2k)$. Then $[X+Y|X\in A_1\cup A_2]=[X+Y|A_X((r,1))]$  takes values on $\{-1,1\}$.
    On the other hand, note that 
\begin{align*}
      A_X((0,r)) &= \big\{X\in \cup_{k=0}^{2m} (k,k+r)\big\}
      = A_Y((-r,0)) = \big\{Y\in \cup_{k=0}^{2m} (-k-r,-k)\big\}. 
\end{align*}
    Then on the set $A_X((0,r))$, 
    $X+Y$ could not only take values on $\{-1,1\}$. There is at least one $k\in\{0,\ldots,2m\}$  such that $X+Y=0$ on $\{X\in (k,k+r)\}$. Hence, $p_1\ge (2m+1)r/(2m+1)T=r/T$. We next show that $r/T$ can be attained by $p_1$ for $F\in \D_2({\rm U}[0,T],{\rm U}[-T,0])$. It suffices to let $[Y|X\in (2m,2m+r)]=-X$ a.s. and for $k=0,\ldots,m-1$
     $$
      [Y|X\in(2k,2k+r)] =-X-1,~~[Y|X\in(2k+1,2k+1+r)] =-X+1~{\rm a.s.}
     $$
  In this case, we have $p_1=r/T.$
%
By symmetry, we can also get the same result if $T=2m-r$. {To see this, note that almost surely
\begin{align*}
      A_X((0,1-r)) &= \big\{X\in \cup_{k=0}^{2m-1} (k,k+1-r)\big\}\\
      &= A_Y((r-1,0)) = \big\{Y\in \cup_{k=0}^{2m-1} (-k-1+r,-k)\big\}. 
\end{align*}
We let
\begin{align*}
     [Y|X\in A_1] = -X-1~~{\rm and}~~[Y|X\in A_2] = -X+1~~{\rm a.s.}
\end{align*}
    where $A_1=\cup_{k=0}^{m-1}  (2k,2k+1-r) $ and $A_2= \cup_{k=0}^{m-1} (2k+1,2k+2-r)$. Then $[X+Y|X\in A_1\cup A_2]=[X+Y|A_X((0,1-r))]$  takes values on $\{-1,1\}$.
 On the other hand, note that almost surely
 \begin{align*}
      A_X((1-r,1)) &= \big\{X\in \cup_{k=1}^{2m-1} (k-r,k)\big\}\\ 
      &= A_Y(\{(-1,r-1)\})= \big\{Y\in \cup_{k=1}^{2m-1} (-k, r-k)\big\}. 
\end{align*}
    Then on the set $A_X((1-r,1))$, 
    $X+Y$ could not only take values on $\{-1,1\}$. There is at least one $k\in\{1,\ldots,2m-1\}$  such that $X+Y=0$ on $\{X\in (k-r,k)\}$. Hence, $p_1\ge (2m-1)r/(2m-1)T=r/T$. We next show that $r/T$ can be attained by $p_1$ for $F\in \D_2({\rm U}[0,T],{\rm U}[-T,0])$. It suffices to let $[Y|X\in (2m-1-r,2m-1)]=-X$ a.s. and for $k=1,\ldots,m-1$
     $$
      [Y|X\in(2k-r,2k)] =-X-1,~~[Y|X\in(2k+1-r,2k+1)] =-X+1~{\rm a.s.}
     $$
  In this case, we have $p_1=r/T.$
}
 Hence, we have when $b$ is an integer,  $Z\in \D_2({\rm U}[0,T],{\rm U}[-T,0])$  if and only if $b=1$ and $\p(Z=-1)=\p(Z=1)\le 1/2-r/(2T)$.

 \item [(b)] If $b$ is not an integer and $T$ is a even integer, without loss of generality, assume $b\in(0,1)$ as the case $b\in (1,2)$ can be discussed similarly by considering  the symmetric distribution on $\{-b,1-b,2-b\}$ and noting that $\D_2({\rm U}[0,T],{\rm U}[-T,0])$ is closed under symmetric transform.  To make sure $\E[Z]=0$, we have $p_0-p_2=1-b$.
     By Part (i), we know the distribution, denoted by $F_1$, on $\{b-1,b\}$ with mean $0$ belongs to $\D_2({\rm U}[0,T],{\rm U}[-T,0])$ which is also closed under mixture. Hence, we have $F\in \D_2({\rm U}[0,T],{\rm U}[-T,0])$ implies
     \begin{align*}
      \lambda F + (1-\lambda) F_0 & =\lambda p_2 \delta_{b-2} + (\lambda p_1 + (1-\lambda) b)\delta_{b-1} +  (\lambda p_0 + (1-\lambda) (1-b) )\delta_{b} \\
      & =\lambda p_2 \delta_{b-2} + (\lambda p_1 + (1-\lambda) b)\delta_{b-1} +  (\lambda p_2 + 1-b )\delta_{b} \\
      & \in \D_2({\rm U}[0,T],{\rm U}[-T,0])~~~~{\rm for~any}~\lambda\in [0,1],
     \end{align*}
   where the second equality follows from $p_0-p_2=1-b$. Therefore, we only need to find the smallest  and the largest values of $p_1$ such that $p_0+p_1+p_2=1$, $p_0-p_2=1-b$ and $F\in  \D_2({\rm U}[0,T],{\rm U}[-T,0])$.

     Note that $F\in  \D_2({\rm U}[0,T],{\rm U}[-T,0])$ is equivalent to there exist $X\sim {\rm U}[0,T] $ and $Y\sim{\rm U}[-T,0]$ such that $Z=X+Y$ a.s. Then by $A_X(L)=A_Y(b-L)$ a.s. for any $L\subset \R$, we have 
     \begin{align*}
       A_X((b,1))
       & = \big\{X\in \cup_{k=1}^{T} (b+k-1,k)\big\}\\
       & = A_Y(\{(b-1,0)\})= \big\{Y\in  \cup_{k=1}^T (b-k,1-k) \big\}.
     \end{align*}
     It is easy to verify that 
      $\E[X+Y| A_X((b,1)) ]=b$. Note that the $X+Y$ takes values on $\{b-2,b-1,b\}$. Then $[X+Y|A_X((b,1))]=b$ a.s. This implies on $A_X((b,1))$, $X+Y$ is not equal to $b-1$ a.s., which in turn implies $b$ is an upper bound  of $p_1$.
      On the other hand, note that 
     \begin{align*}
       A_X((0,b)) &= \big\{X\in  \cup_{k=1}^T (k-1,b+k-1) \big\}\\ &= A_Y(\{(0,b)\})= \big\{Y\in \cup_{k=1}^T (-k,b-k) \big\}.
     \end{align*}
     Letting
     $[Y| A_X((0,b))] = b-X-1$ a.s. and $[Y| A_X((b,1))] = b-X$ a.s.
    yield that $b$ is the largest value of $p_1$.

     To find the smallest  value of $p_1$, we consider two cases that $T$ is even and odd. If  $T$ is even, then $T=2m$ for some $m\in\N$.  On the set $A_X(0,b)$, for $k=0,1,\ldots,m-1$, we let
     $$
      [Y|X\in (2k,b+2k)] = b-X-2
     ~~{\rm and}~~
      [Y|X\in (2k+1,b+2k+1)] = b-X,~{\rm a.s.}
     $$
    In this case, $p_1$ is zero  which is the   smallest possible value of $p_1$. 

    \item [(c)] If $b$ is not an integer and $T$ is an odd integer, then similar to Part (b), we only need to find the largest and the smallest value of $p_1$ for the case $b\in (0,1)$.  The largest value of $p_1$ is $b$.  To find the smallest value, note that $T=2m+1$ for some $m\in \N$.  On the set $A_X(0,b)$, we let $[Y|X\in (2m,b+2m)] = b-X$ a.s. and  for  $k=0,1,\ldots,m-1$,
     $$
      [Y|X\in (2k,b+2k)] = b-X-2,~~~
      [Y|X\in (2k+1,b+2k+1)] = b-X,~{\rm a.s.}
     $$
In this case,  $p_1=b/T$  which is the smallest possible value of $p_1$. This is due to $ [X+Y|X\in (2m,b+2m)] >-1>b-2$ as $\{X\in (2m,b+2m)\}\subset A_Y(0,b)$ and $[Y|A_Y((0,b))]\ge -2m-1$ a.s.
\end{itemize}
Combining the above three cases, we complete the proof for the case of $k=2$. \qedhere
\end{itemize}
\end{proof}

\begin{proof} [Proof of Theorem \ref{pr-contines}]
Note that $3b/4h \times 2h =3b/2 \le1$, hence $b\le 2/3$. Therefore, $[a,a+b]\subset [0,2]$.
Denote  $\lambda=3b/2\ge 3h$ and  let $H$ denote the distribution function of ${\rm U}[1-h,1+h]$.  Define $G = (F - \lambda H)/(1-\lambda)$. By the assumption on $F$,   $G$ is also a distribution with positive density and mean $1$.  There exists a sequence of distribution functions $\{G_n\}_{n\ge 2}$ with finite support on $[a,a+b]$  and mean $1$ which converges to $G$ in distribution as $n\to\infty$.
Define $F_n = \lambda H+(1-\lambda)G_n$, $n\ge 2$.  Then $F_n$  converges to $F$ is distribution as $n\to\infty$. Note that ${ \D}_2^{\rm U} $ is closed with respect to weak convergence by Lemma \ref{lm-180905-1}. We only need to  show $F_n\in { \D}_2^{\rm U} $ for $n\ge 2$. Without loss of generality, assume $$
 G_n(\{x_i\}) = p_i,~~i=1,\ldots,n~{\rm with}~~p_1+\cdots+p_n=1,~a\le x_1<\cdots<x_n\le a+b.
$$
We show it by induction on $n$.
Note that $a\le 1-h\le 1+h\le b$. For $n=2$, without loss of generality,  let $c,d>0$ be such that $x_1=1-c$ and $x_2=1+d>0$, $3(c+d)/2\le\lambda$ and then
\begin{equation}\label{eq-181012-1}
 G_2(\{1-c\}) = \frac d{c+d}~~{\rm and}~~G_2(\{1+d\}) = \frac c{c+d}.
\end{equation}
Since $ 1 \ge  \lambda \ge c + d + h $, we have  $$  \frac{ 1- h }{c+d}  - \frac{1 - \lambda } {c+d} \ge 1 ~~{\rm and}~~   \frac{ 1- h }{c+d} \ge 1.$$
Then there exists  some integer $k \ge 1 $ such that   $$  \frac{ 1- h }{c+d}  \ge k \ge  \frac{1 - \lambda } {c+d}    .$$
  That is,
$ 1-  k (c+d)   \ge  h  ~ \text{and} ~ k(c+d)  \ge 1-\lambda .$
Denote  $$ \theta :=   \frac {  k(c+d) +  \lambda -1  }{ k(c+d) } \in [0, \lambda]~~ {\rm as}~~\lambda - \theta =  \frac {  (1-\lambda)(1-k(c+d)) }{ k(c+d) } \ge 0 . $$
Note that  the length of the support of $ \U [0,k(c+d) ] $ ($\U[ 1-k(c+d), 1 ] $)  is a multiple of $(1+d)-(1-c) $. By Lemma \ref{lm-180228-1} (i), we have  $G_2\in \mathcal D_2 ( \U [0,k(c+d) ] , \U[ 1-k(c+d), 1 ] ) $.
Similarly, we have $ H \in \mathcal D_2 ( \U[k(c+d) ,1], \U[ 0 , 1-k(c+d)]) $.
 Also, 
 by Theorem \ref{main-th}, we know $H \in \mathcal D_2^{\rm U} $.  It follows that $$ F_2=( 1- \lambda ) G_2 +  ( \lambda - \theta  )  H +  \theta   H \in \mathcal D_2^{\rm U}.$$
Suppose that $F_n=( 1- \lambda ) G_n +  \lambda H \in \mathcal D_2^{\rm U}$ for $n\le k$ and we aim to show $F_{k+1}=(1- \lambda) G_{k+1} +  \lambda H \in \mathcal D_2^{\rm U} $.
Let $G_{k+1,1}$ be defined by \eqref{eq-181012-1} with $c=1-x_1>0$ and $d=x_{k+1}-1>0$. That is,
$$
 G_{k+1,1}(\{x_1\}) = \frac {x_{k+1}-1}{x_{k+1}-x_1}~~{\rm and}~~G_{k+1,1}(\{x_{k+1}\}) = \frac {1-x_1}{x_{k+1}-x_1}.
$$
Let $\alpha := \min\{p_1(c+d)/d,p_n(c+d)/c\}$. Then  $G_{2:k+1} :=  (G_{k+1} - \alpha G_{k+1,1})/(1-\alpha)$ is a distribution function  with support on $\{x_1,\ldots,x_{k}\}$ or $\{x_2,\ldots,x_{k+1}\}$ and mean $1$. Then we have
\begin{align*}
 F_{k+1} & =  ( 1- \lambda ) (\alpha G_{k+1,1} + (1-\alpha) G_{2:k+1}) +  \lambda H \\
 & = \alpha (( 1- \lambda ) G_{k+1,1} +  \lambda H)+(1-\alpha)(( 1- \lambda ) G_{2:k+1}+  \lambda H)).
\end{align*}
By induction, we have $( 1- \lambda ) G_{k+1,1} +  \lambda H \in \mathcal D_2^{\rm U}$ and $( 1- \lambda ) G_{2:k+1}+  \lambda H\in \mathcal D_2^{\rm U}$. Then by the convexity of $\mathcal D_2^{\rm U}$ from Lemma \ref{lm-180905-1}, we have $F_{k+1} \in \mathcal D_2^{\rm U}$. Thus, we complete the proof.
\end{proof}

%
%
%
}


\begin{proof}[Proof of Theorems  \ref{th-bi} and \ref{th-tri}]
{We only give the proof of Theorem \ref{th-tri} as Theorem \ref{th-bi} can be proved similarly based on Lemma \ref{lm-180228-1} (ii).} Note that the cumulative distribution function of a random variable $Z$ belongs to $\mathcal D_2^{\mathrm U}$ if and only if the cumulative distribution function of $T(Z-1)$ belongs to $\mathcal D_2({\mathrm U}[0,T], {\mathrm U}[-T,0])$ {for $T\ge 0$.     
Let  $Z$ be a random variable with distribution $F$ and let $T:=1/b$ and $c := (a+b-1)/b$, that is, $b=1/T$ and $a=c/T+1-1/T$. Then $Z_T:=T(Z-1)$ is a random variable  satisfying 
$$
\p(Z_T=c-i)=f_{i+1},~~i=0,1,2. 
$$ 
By Lemma \ref{lm-180228-1}  (iii), we know  the cumulative distribution function of  $Z_T$ belongs to $\mathcal D_2^{\mathrm U}$ if and only if one of the three cases of Lemma \ref{lm-180228-1}  (iii)  holds  by replacing $b$ by $c$ and $T$ by $1/b$ respectively. That is,
   \begin{itemize}
\item [(a)] $a+b-1/b=1$, $f_2\ge rb$ with $1/b=2m\pm r$ with $m\in\N$ and $r\in (-1,1)$.  
\item [(b)]  $(a+b-1)/b<1$,  $1/b$ is even, $f_2\in (0, (a+b-1)/b)$, $f_3=(c-f_2)/2$ and $f_1=1-((a+b-1)/b+f_2)/2$.
\item [(c)]  $(a+b-1)/b<1$,  $1/b$ is odd, $f_2\in (a+b-1,(a+b-1)/b)$, $f_3=(c-f_2)/2$  and $f_1=1-(c+f_2)/2$.
\end{itemize}
Note that  under the constraint $0<b\le a\le 1$, we have $c\le 1$;  $c=1$ is equivalent to $a=1$;  $rb=\lceil 1/(2b) \rfloor$; $1/b$ is even  if and only if  $1/({2b})\in \N$; $1/b$ is odd  if and only if $1/({2b}) -1 /2 \in \N$. Also, by $\E[Z]=1$, we have $f_3=f_1+(1-a)/b$ and thus, $f_2=1-2f_1-(1-a)/b$ which implies that  $f_2\in (0, (a+b-1)/b)$ always holds. 
 Hence, the statement in Theorem \ref{th-tri} holds.
}
\end{proof}

\section{Sums of three or more standard uniform random variables}\label{sec:4}
In this section, we aim to show that for $n \ge 3$, the two sets $\mathcal D_n^{\rm U}$ and $\mathcal C^{\rm U}_n$ are identical, in sharp contrast to the case of $n=2$ analyzed in Section \ref{sec:3}.
We start with the  bi-atomic distribution. Let
$F$ be the distribution function of a random variable $X$ such that
$\p(X=a)=p$ and $\p(X=b)=1-p$ with $\E[X]=1/2$, $a<b$ and $0<p<1$ and $T_{n} (F)$ be the   distribution function of $nX$. That is,
$$
 F = p \delta _{a} + (1- p) \delta _{ b }~~{\rm and}~~T_{n} (F) = p\delta _{na} + (1- p) \delta _{n b},
$$
 where $\delta_x$ denotes the point-mass at   $x\in \R$.


\begin{lemma}\label{lm-eqbi}
Let   $F $ be a bi-atomic distribution on $\{a,b\}$ with $a<b$. Then the following statements are equivalent.
 \begin{enumerate}[(i)]
 \item $F\lcx \U[0,1]$.
 \item $b-a\le 1/2$ and $\mu(F)=1/2$. 
 \item
 $T_{n} (F) \in \mathcal D_n^{\U}(1,\dots,1)$ for $n\ge 3$.
 \end{enumerate}
\end{lemma}

\begin{proof} It is easy to verify  (iii) $\Rightarrow$ (i). It suffices to show (i) $\Rightarrow$ (ii) $\Rightarrow$ (iii).  Let $X$ be a random variable having distribution $F$ {and  $\p(X=a)=p=1-\p(X=b)$.  Note that under the constraint of (i), we must have $a,b\in (0,1)$ and $\mu(F)=1/2$ which implies $0<a<1/2<b<1$; under the constraint of (ii), by $\mu(F)=1/2$, we have $a<1/2<b$. Then by  $b-a\le 1/2$,  we also have $a>0$ and $b<1$. Thus, in the following proof, we always assume $0<a<1/2<b<1$.

Without loss of generality,  we also assume $ 0 < a \le 1-b \le 1/2$  by symmetry. Otherwise, consider another random variable $X^*=1-X$ with distribution $F^*$ and it suffices to note that  $\mu(F^*)=1/2$, $b^*-a^*=b-a$., and $F\lcx \U[0,1]$ is equivalent to $F^*\lcx \U[0,1]$. }

(i) $\Rightarrow$ (ii):   By $F\lcx \U[0,1]$, we have   $\E[X]=1/2$ and 
    $
     \E[(X-t)_+] \le {(1-t)^2}/{2},
    $
   for $t\in [0,1]$, that is,
   \begin{equation}\label{eq-cxcon}
       (b-t)(1-p)\le  \frac{(1-t)^2}{2},~~a\le t\le b.
    \end{equation}
   To show (ii), {we only need to show $b-a\le 1/2$, that is,  the largest possible value of $b-a$ is $1/2$. To do this, we fix the value of $p$ and denote $\ell_1(t):= (b-t)(1-p)$, $t\in [a,b]$ and $\ell_2(t):=  {(1-t)^2}/{2}$, $t\in [0,1]$.   Note that $\ell_1(t)$ is a linear function with fixed and constant derivative  equaling to $p-1$  and lies below the quadratic curve $\ell_2(t)$ for $t\in [a,b]$. Also note that as $a$ decreases or $b$ increases, $\ell_1$ moves upwards. Hence, when $b-a$ attains its largest possible value, we have the line $\ell_1$ and the quadratic curve $\ell_2$ are tangent and the tangent point satisfies  the equation $\ell_2'(t)=t-1=p-1$, that is, the tangent point is $t=p$. Then by $\ell_1(p)=\ell_2(p)$, we get  $b=1+p/2$ which is  the largest possible value  of $b$. Then by $\E[X]=1/2$, we can get $a=p/2$, which is the smallest possible value of $a$. Therefore, we have the largest possible value of $b-a$ is $1/2$.
}    

(ii) $\Rightarrow$ (iii): Let $G_{x,y}$ be the distribution of $ \U[x,y]$. It suffices to show that $G_{0,1}= (1-p) H_1 + p H_2$ such that both $H_1$ and $H_2$ are $n$-CM, $\mu(H_1)=a$ and $\mu(H_2)=b$.  
We first define two distributions $H_1^*$ and $H_2^*$ for the following three cases with $1-b\ge a$ in mind.
\begin{itemize}
\item [(a)] If $a \ge 1/n$, define $H_1^* = H_2^* =  G_{0,1}$, that is, the distribution function of $\U[0,1]$.

\item [(b)] If $1- b \ge 1/n \ge a $, define
$$ H_1^* = G_{0,na}~~~{\rm and}~~~ H_2^* =  \frac{na -p}{1-p} G_{0,na}+ \frac{1-na}{1-p} G_{na,1}.$$
By $\E[X]=1/2$, we have $$ p = \frac{b-1/2 }{b-a}  = 1- \frac{1/2-a}{b-a} \le 1- 2(\frac12-a) = 2a < na, $$
where the inequality follows from $b-a\le 1/2$. Hence, $H_2^*$ is a distribution with positive density function. It is obvious that $ \mu(H_1^*) = {na}/{2} \ge a $ and thus $\mu (H_2^*) \le b$ as $p\, \mu (H_1^*) + (1-p) \mu(H_2^*) = 1/2$ and $pa+(1-p)b=1/2$.

\item [(c)] If $1-b< 1/n$, then we have $n=3$ as $b-a\le 1/2$ and $a\le 1-b$.  Define
$$ H_1^* = \frac{3b-2}{p } G_{0,3b-2} + \frac{p   +2-3b}{p } G_{3b-2, 3a}~~{\rm and}~~H_2^* =  \frac{ 3a -p }{1-p} G_{3b-2, 3a} +  \frac{1-3a}{1-p} G_{3a,1}.$$
 Note that $ p = \frac{b-1/2 }{b-a}  \ge 2(b-1/2)   >  3b-2 $ and $p < 3a$. Hence, we have both $H_1^*$ and $H_2^*$ are distribution functions with positive densities.
  It is easy to calculate that
$$ \mu(H_1^*)
= \frac{3a + 3b}{2} -1 +\frac{3a }{p} -\frac{ 9 ab   }{2p}. $$
Note that
 \begin{align*}
 2p(\mu(H_1^*) - a ) & = p(3b + a - 2)  + 3a(2 - 3b)\\
 & \ge (3b- 2) (3b + a - 2) - 2(b-a)3a(3b-2)\\
 & = (3b-2)( 3b+a -2 +2(b-a)3a )\\
 & \stackrel{\rm sgn}=3b - 2 + a + 2(b-a)3a\ge 0,
 \end{align*}
where the first inequality follows from $p>3b-2$ and $2(b-a)\le 1$,  
 $A\stackrel{\rm sgn}=B$ represents that $A$ and $B$ have the same sign, and {the $\stackrel{\rm sgn}=$ is due to $1-b< 1/n$ and $n=3$ by the observations at the beginning of (c).}
Hence, we have $\mu(H_1^*)\ge a$ and similarly $\mu(H_2^*)\le b$.
\end{itemize}
In each of the above three cases, we have $G_{0,1}= p H_1^* + (1-p)H_2^*$,  $H_1^*$ has a decreasing density on $[0,na]$ with $\mu(H_1^*) \ge a$, and $H_2^*$ has an increasing density on $[1-n(1-b),1]$ with $\mu(H_2^*) \le b$.

 On the other hand, let $H_1^0= G_{0,p}$ and  $H_2^0 = G_{p,1}$. Then it is obvious that  $G_{0,1}=pH_1^0 + (1-p)H_2^0$,  $\mu(H_1^0) \le a $ and $\mu(H_2^0)\ge b$.
Hence,  we can find some $\alpha \in [0,1]$ such that $ \alpha \mu(H_1^*) + (1-\alpha) \mu(H_1^0) = a$.
Then  the distribution
$H_1:= \alpha H_1^* + (1-\alpha)H_1^0 $ is supported in $[0,na]$ with decreasing density and $\mu(H_1)=a$. By Theorem 3.2 of \cite{WW16}, $G$ is $n$-CM, that is, $\delta _{na} \in \mathcal D_n( G,\dots,G )$.
Similarly, we have $H_2 = \alpha H_2^* + (1-\alpha)H_2^0 $ is supported in $[1-n(1-b),1]$ with an increasing density and $\mu(H_2)=b$, which implies that $H_2$ is also $n$-CM. That is, $ \delta_ {nb} \in \mathcal D_n( H_2,\dots,H_2)$.
By Lemma \ref{lm-180905-1}, we have
 $$ T_{n} (F) =  p \delta _{n a} + (1- p) \delta _{n b }  \in \mathcal D_n( p H_1 + (1-p )H_2,\dots, p H_1 + (1-p )H_2) = \mathcal D_n( \U[0,1] ,\dots, \U[0,1]).$$
Thus, we complete the proof.
\end{proof}

Now we are ready to show our main result in dimension $n\ge 3$.
It turns out that for standard uniform distributions, the two sets
$\mathcal D_n(F_1,\dots, F_n) $ and $ \mathcal C(F_1\oplus\dots\oplus F_n)$ in Lemma \ref{lm-cxset-1} coincide.
The following theorem is, to the best of our knowledge, the first analytical characterization of $\mathcal D_n$ for  continuous marginal distributions.
\begin{theorem}\label{th:main} For $n \ge3$, we have $\mathcal D_n^{\rm U} = \mathcal C^{\rm U}_n$.
\end{theorem}
\begin{proof} As $\mathcal D_n^{\rm U} \subseteq \mathcal C^{\rm U}_n$, it suffices to show $  \mathcal C^{\rm U}_n\subseteq \mathcal D_n^{\rm U},$ that is, for any distribution function $F$,
  $F\lcx \U[0,1]$ implies $\mathcal D_n^{\rm U}(1/n)$. Here and throughout the proof, we use the notation $ \mathcal D_n^{\rm U}(x) =  \mathcal D_n({\rm U}[0,x],\ldots,{\rm U}[0,x])$ for $x\in\R.$  For a distribution $F$, denote
$$ W_F(t) =  \E [(X-t)_+]  -  \frac{(1-t)^2}{2},~~ t \in [0,1] ,$$
where $X$ is a random variable having the distribution function $F$.
 We first consider the special case that  $F$ is a distribution function of a discrete random variable $(a_1,p_1;\ldots;a_m,p_m)$ with $0\le a_1<\cdots<a_m\le 1$ and $G_{x,y}$ is the distribution function of $\U[x,y]$, $x<y$.  By Lemma \ref{lm-eqbi}, we know the result holds for $m=2$. Next, we show it holds for general $m\ge 2$ by induction.

 For general $m\ge 2$, 
%
%
by $F\prec_{\rm cx} G_{0,1}$, we have for 
$W_F(t)\le 0$ for $t\in[0,1]$,
    that is, for $k=2,\ldots,m$, we have
   \begin{equation}\label{eq-cxcon}
       (a_k-t)p_k+\cdots+(a_m-t)p_m\le \frac{(1-t)^2}{2},~~~a_{k-1}\le t\le a_{k}.
    \end{equation}
    Next, we consider two cases.
    \begin{enumerate}[(a)]
\item    If there exists $t\in [a_{k-1}, a_{k})$ such that the equality  of \eqref{eq-cxcon} holds, then $p_1+\cdots+p_{k-1}=t$. {To see this, denote $\ell_1(t):= (a_k-t)p_k+\cdots+(a_m-t)p_m$, and $\ell_2(t)={(1-t)^2}/{2}$, $t\in [a_{k-1},a_k]$.
Note that $\ell_1$ is a linear function which lies below the decreasing quadratic curve$\ell_2$. Thus,   $\ell_1$ and the quadratic curve $\ell_2$ are tangent at the point $t$ and the tangent point $t$ satisfies  the equation $t-1=  -p_k-\cdots-p_m$, that is,  $p_1+\cdots+p_{k-1}=t$.}
Let $X_1$ and $X_2$ be two random variables satisfying $\p(X_1=a_i)=p_i/t$, $i=1,\ldots,k-1$, and $\p(X_2=a_i)=p_i/(1-t)$, $i=k,\ldots,m$. Denote by $F_1$ and $F_2$  the distributions of $X_1$ and $X_2$, respectively. It is easy to verify that $F=tF_1+(1-t)F_2$, $F_1\prec_{\rm cx} G_{0,t}$ and $F_2\prec_{\rm cx} G_{t,1}$.
   Then by induction, we have $F_1\in  \mathcal D_n^{\rm U}(t/n)$ and  $F_2\in   \mathcal D_n({\rm U}[t/n,1/n],\ldots,{\rm U}[t/n,1/n])$. That is, there exist $X_{11},\ldots,X_{1n}\sim {\rm U}[0,t/n]$ and $X_{21},\ldots,X_{2n}\sim {\rm U}[t/n,1/n]$ such that
   $$
       X_{11}+\cdots+X_{1n}\sim F_1~~{\rm and}~~ X_{21}+\cdots+X_{2n}\sim F_2.
   $$
Without loss of generality, assume  that   $X_{11},\ldots,X_{1n}$ are independent of $X_{21},\ldots,X_{2n}$.   Let $A$ be a random event independent of $X_{11},\ldots,X_{1n}$, $X_{21},\ldots,X_{2n}$ such that $\p(A)=p_1+\cdots+p_{k-1}=t$. Define
   $$
      Y_{i} =  X_{1i}1_A + X_{2i} 1_{A^c},~~~~i=1,\ldots,n.
   $$
   It is obvious that  $Y_i\sim {\rm U}[0,1]$, $i=1,\ldots,n$, and $Y_1+\cdots+Y_n\sim F$. This means $F\in  \mathcal D_n^{\rm U}(1/n,\ldots, 1/n)$.

 \item If the  inequality  of \eqref{eq-cxcon} is strict for every $t\in (0,1)$, define two functions $G$ and $H$ as
$$ G =  \lambda \delta_{a_1} +(1-\lambda)\delta_{a_m} ~~\text{ and} ~~H = \frac{F - \theta G }{1-\theta},$$
where $\delta_{a}$ denote the degenerated distribution at point $a$,
$$
\lambda=\frac{a_m - 1/2 } {a_m-a_1} ~~{\rm and}~~\theta = \min \left\{ \frac {p_1} \lambda, \frac{p_m} {1-\lambda}\right\}< 1.
$$
It is easy to verify that $G$ and $H$ are two distribution functions satisfying $ F = \theta G + (1-\theta) H $. We also assert that  $ H \lcx F$.  To see it, let $X\sim F,$ $Y \sim G ,$ and $Z \sim H $. Then for any convex function $\phi$, we have
  $$\phi(X) \le   \frac{a_m - X }{ a_m  - a_1  }  \phi(a_1) +  \frac{ X-a_1 }{ a_m  - a_1  }  \phi(a_m) ,~~{a.s.},$$
It then follows that
$$ \E[\phi(X)] \le  \frac{a_m - \E [X] }{ a_m  - a_1  }  \phi(a_1) + \frac{ \E[X]-a_1 }{ a_m  - a_1  } \phi(a_m) =
\lambda \phi(a_1) + (1-\lambda) \phi(a_m) = \E [\phi(Y)], $$
Also, note that  $F = \theta G + (1-\theta) H$ which implies
$ \E[\phi(X)] = \theta \E[\phi(Y)] + (1- \theta ) \E[\phi(Z)]$. Combined with $ \E[\phi(X)] \le  \E[\phi(Y)]$, we have  $\E[\phi(Z)] \le \E[\phi(X)],$
that is, $ H \lcx F$. This implies $H\lcx \U[0,1]$ and note that the support of distribution  $H$ has  at most $m-1$ points. Using inductive hypothesis, we have
$H \in \mathcal D_n^{\rm U}(1/n)$.


On the other hand, if $ a_m - a_1 \le \frac{1}{2} $, we know $ G \lcx \U[0,1] $,  $G \in \mathcal D_n^{\rm U}(1/n,\ldots, 1/n)$. Then
$F = \theta G + (1-\theta) H \in \mathcal D_n^{\rm U}(1/n)$. Otherwise, if $ a_m - a_1 > \frac{1}{2} $, then $ G \not\lcx \U[0,1]$, that is, $ W_G(t_0) > 0 $ for some $t_0\in(0,1)$.
 Define a function $$\alpha(t) = -\frac{ W_{G} (t) }{  W_{F} (t)} ,~~ t\in (0,1) $$
 which is continuous function satisfying $\alpha (0+) = \alpha (1- ) = - 1 $ as $ W_F(t) = W_G(t) = - {t^2}/{2}$ for $t\in [0,a_1] $  and $W_F(t) = W_G(t) = -{(1-t)^2}/{2}$ for $t\in [a_m,1] $.
  Hence, we have $\alpha(t) $ takes its maximum value  at some $t\in (0,1)$ and the maximum value is positive.  Without loss of generality, assume  $\alpha_0 =  \alpha (t_1) > 0 $ is its maximum value. Then we have
 \begin{equation}\label{eq-180905-1}
 W_G (t) + \alpha_0 W_F(t) \le 0  ~~{\rm for~all}~ t\in [0,1]~~  \text {and} ~~  W_G (t_1) + \alpha_0 W_F(t_1) = 0.
 \end{equation}
Define a distribution $F_0 := (G + \alpha_0 F )/({1+ \alpha_0})$. Then
by \eqref{eq-180905-1}, we have
$F_0 \lcx \U[0,1]$ and $W_{F_0} (t_1 ) = 0 $. By Case (a),  we have $F_0  \in \mathcal  D_n^{\rm U}(1/n)$.

Note that  $ G = (1+\alpha_0)  F_0 - \alpha_0 F $ and $F = \theta G + (1-\theta) H $, which implies
$ F = (1+\alpha_0)  \theta F_0 - \alpha_0 \theta F + (1-\theta) H, $
that is,
$$F= \frac   {  ( \theta+ \alpha_0 \theta )  F_0 + (1-\theta) H   }{  1 +  \alpha_0  \theta  }. $$
Then by Lemma \ref{lm-180905-1}, we have $F  \in \mathcal  D_n^{\rm U}(1/n)$.

%
%
%
%
%
   \end{enumerate}
If $F$ is a general distribution function such that $F\prec_{\rm cx} {\rm U}[0,1]$, then  ${\rm Supp}(F)\subset [0,1]$ and it has no mass on $0$ and $1$.  For any  $n\in\N$, define $F_n$ as the distribution function of $X_n$
 $$
    X_n= \sum_{k=1}^{n} \E \left[X \left| \frac{k-1}{n} \le X < \frac{k}{n} \right.\right] 1_{\{ \frac{k-1}{n} \le X < \frac{k}{n} \}}
 $$
where $X$ is a random variable having distribution function $F$.  Then $F_n$ converges to $F$ in weak convergence as $n\to\infty$, and  $F_n\prec_{\rm cx} {\rm U}[0,1]$.
By the above proof for discrete disitributions with finite support, we have $F_n\in \mathcal  D_n^{\rm U}(1/n)$ for each $n\in\N$. Then by Lemma \ref{lm-180905-1}, we have $F\in \mathcal  D_n^{\rm U}(1/n)$. Thus, we complete the proof.
\end{proof}

For any random variable $X\sim F$ with mean 0, we have $F_a\lcx F$ for any $a\in [0,1]$, where $F_a$ is the distribution of $aX$. Hence, we immediately get the following corollary.
Note that this corollary, although looks simple, does not seem to allow for an elementary proof without using Theorem \ref{th:main}.

\begin{corollary}
For $n \ge 3$, if  $F\in\mathcal D_n (\U[-1,1],\ldots,\U[-1,1])$, then so is $F_a$ for all $a\in [0,1]$.
\end{corollary}

\section{An application}

In risk management,
often one needs to optimize a  statistical functional, mapping $\mathcal F$ to $\R$ (such as a risk measure), over the set of $\mathcal D_n(F_1,\dots,F_n)$,
and this type of problem is called \emph{risk aggregation with dependence uncertainty} (see e.g.~\cite{EPR13} and \cite{BJW14}).
These problems are typically quite difficult to solve in general, as the set $\mathcal D_n(F_1,\dots,F_n)$ is a complicated object.
For uniform marginal distributions, using results in this paper (in particular, Theorem \ref{th:main}), we are able to translate many optimization problems on $\mathcal D_n^{\mathrm{U}}$
to $\mathcal C^{\rm U}_n$ for $n\ge 3$, which is a convenient object to work with.

We study an application of the problem of minimizing or maximizing for a given interval $A$, the value of $\p(S\in A)$ where $S$ is the sum of $n$ standard uniform random variables.
A special case of the problem concerns bounds on $\p(S\le x)$ for $x\in \R$, i.e., bounds on $F(x)$ for $F\in  \mathcal D^{\mathrm U} _n $,  is studied by \cite{R82}.
Using  Theorem \ref{th:main}, we are able to solve the problem of $\p(S\in A)$ completely.

\begin{proposition}\label{prop:51}
For $n\ge 3$, and $0\le a\le a+b\le n$,
we have
\begin{align}\label{eq:new1}
\min_{F_S\in \mathcal D^{\mathrm U}_n} \p(S\in (a,a+b)) = \left(\frac{2b}n-1\right)_+,
\end{align}
and
 \begin{align}\label{eq:new2}
\max_{F_S\in \mathcal D^{\mathrm U}_n }\p(S\in [a,a+b])  = \min\left\{\frac{2(a+b)}n, \frac{2(n-a)}n,1\right\},
\end{align}
where $F_S$ stands for the cdf of $S$.
\end{proposition}

\begin{proof}
As Theorem \ref{th:main} gives $ \mathcal D^{\mathrm U} _n=\mathcal C^{\rm U}_n$, it suffices to look at the optimization problems for $ \mathcal C^{\rm U}_n$.
For $0\le u\le v\le n$, let $\mathcal A_{u,v}$ be the sigma field generated by $\{U\le u\}$ and $\{U\le v\}$.
 $S=\E[U|\mathcal A_{u,v}]$ is  tri-atomically distributed
with distribution measure \begin{align}\label{eq:new3} \frac un \delta_{u/2}  + \frac{v-u}{n}\delta_{(u+v)/2}+ \frac{n-v}{n}\delta_{(n+v)/2}.\end{align}
Note that $F_S\prec_{\rm cx} F_U$  because  $S$ is a conditional expectation of $U$, and thus $F_S \in \mathcal C^{\rm U}_n= \mathcal D^{\mathrm U} _n$.

We first verify the following inequality (indeed, it is Theorem 1 of \cite{R82}). For any $S$ which is the sum of $n$ standard uniform random variables and $x\in \R$, we have
\begin{equation}\label{eq:new4} \p(S\le x)\le 2x/n ~~~\mbox{and }~~~\p(S\ge x)\le 2(n-x)/n.\end{equation}
Equation \eqref{eq:new4} can be shown using the equivalent condition of convex order (see Theorem 3.A.5 of \cite{SS07}). Using $F_S\prec_{\rm cx} F_U$,
\begin{align*}
F^{-1}_S(\alpha)\alpha \ge \int_0^{\alpha} F^{-1}_S(t) \d t \ge \int_0^{\alpha } F^{-1}_U(t)\d t = \frac{n\alpha^2}{2}, ~~ \alpha \in (0,1) .
\end{align*}
Therefore, $F^{-1}_S(\alpha)\ge \frac{n\alpha}{2}$, $\alpha \in (0,1)$, and equivalently $F_S(x)\le 2x/n$, $x\in \R$. The other inequality in \eqref{eq:new4} is symmetric {by noting that $n-S$ is still the sum of $n$ standard uniform random variables and $\p(S\ge x)=\p(n-S\le n-x)$.}

We now analyze the problem of the minimum in \eqref{eq:new1}.
\begin{enumerate}[(i)]
\item Suppose $b\le n/2$. Since $(n+u)/2-u/2 =n/2>b$,
we can find $u\in [0,n]$ such that $u/2\le a$ and $(n+u)/2\ge b+a$.
By letting $S=\E[U|\mathcal A_{u,u}]$ and using \eqref{eq:new3}, we have
 $\p(S=u/2) = u/n$ and  $\p(S=(n+u)/2)=(n-u)/n$.
In this case, $\p(S\in (a,a+b))=0$; thus \eqref{eq:new1} holds.
\item Suppose $b>n/2$, which implies $a<n/2$ and $a+b>n/2$.
Let $S$ be given by $\E[U|\mathcal A_{u,v}]$
where $u=2a$ and $v=2(a+b)-n$.
Note that $\p(S\in (a,a+b))= \frac{v-u}{n} = \frac{2b}n-1.$ This shows the ``$\le $" direction of \eqref{eq:new1}.
On the other hand, by \eqref{eq:new4}, for any $S$ which is the sum of $n$ standard uniform random variables, $\p(S\le a)\le 2a/n$ and  $\p(S\ge a+b)\le 2(n-a-b)/n.$
Thus,
$$\p(S\in (a,a+b))\ge 1-\frac{2a}n - \frac{2(n-a-b)}{n} =  \frac{2b}n-1. $$
This shows the ``$\ge$" direction of \eqref{eq:new1}.
\end{enumerate}
Next, we  analyze the problem of the maximum in \eqref{eq:new2}.
\begin{enumerate}[(i)]

\item If $a+b> n/2$ and $a< n/2$, then $n/2\in [a,a+b]$. Taking $S=\E[U]=n/2$ gives $\p(S\in [a,a+b])=1$; thus \eqref{eq:new2} holds.
\item Suppose $a+b\le n/2$. By \eqref{eq:new4}, for any $S$ which is the sum of $n$ standard uniform random variables,
$\p(S\in [a,a+b])\le \p(S\le a+b) \le 2(a+b)/n$.
To see that such a bound is attainable, take
$S=\E[U|\mathcal A_{u,u}]$ where $u=2(a+b)$. Then, by \eqref{eq:new3}, we have
$$\p(S\in [a,a+b])\ge \p(S=a+b) = \p\left(S=\frac u2\right) =     \frac un = \frac{2(a+b)}n.$$
Therefore, \eqref{eq:new2} holds.
\item Suppose $a\ge n/2$.  Similar to the above case, by \eqref{eq:new4}, for any $S$ which is the sum of $n$ standard uniform random variables,
$\p(S\in [a,a+b])\le \p(S\ge a) \le 2(n-a)/n$.
To see that such a bound is attainable, take
$S=\E[U|\mathcal A_{u,u}]$ where $u=2a-n$. Then, by \eqref{eq:new3}, we have
$$\pi(S\in [a,a+b])\ge \p(S=a) = \p\left(S=\frac {n+u}2\right) =     \frac {n-u}n = \frac{2(n-a)}n.$$
Therefore, \eqref{eq:new2} holds.  \qedhere
\end{enumerate}
\end{proof}

\begin{remark}
Based on the proof of Proposition \ref{prop:51},
we can identify some minimizing distributions for \eqref{eq:new1} and some maximizing distributions for  \eqref{eq:new2}.
For $0\le u\le v\le n$,  write the  distribution 
 \begin{align}\label{eq:new-r1} F_{u,v}= \frac un \delta_{u/2}  + \frac{v-u}{n}\delta_{(u+v)/2}+ \frac{n-v}{n}\delta_{(n+v)/2}. \end{align}
There are a few cases. For the minimum in \eqref{eq:new1}: 
\begin{enumerate}
\item If $b\le n/2$, then  $F_{u,u}$  attains \eqref{eq:new1} for $u \in [2a+2b-2n, 2a]$.
\item  If $b> n/2$, then  $F_{u,v} $ attains \eqref{eq:new1} where  $u=2a$ and $v=2(a+b)-n$.
\end{enumerate}
 For the maximum in \eqref{eq:new2}: 
\begin{enumerate} 
\item If $a+b> n/2$ and $a< n/2$, then $\delta_{n/2}$ attains \eqref{eq:new2}. 
\item If  $a+b\le n/2$,  then  $F_{u,u}$  attains  \eqref{eq:new2} where $u=2(a+b)$.
\item If  $a\ge n/2$, then  $F_{u,u}$  attains  \eqref{eq:new2}   where $u=2a-n$.
\end{enumerate}
\end{remark}

\subsection*{Acknowledgements}
The authors thank the Editor, an Associate Editor, and an anonymous referee for careful reading of our paper and very helpful comments which improve the quality of the paper.
T.~Mao was supported by the NNSF of China (grant numbers:~71671176, 71871208, 11501575). B.~Wang acknowledges financial support from the NNSF of China (grant number:~11501017).
 R.~Wang   acknowledges financial support from the Natural Sciences and Engineering Research Council of Canada (RGPIN-2018-03823, RGPAS-2018-522590), and the Center of Actuarial Excellence Research Grant from the Society of Actuaries.

\end{document}